\documentclass{article}
\usepackage[utf8]{inputenc}
\usepackage{graphicx}
\usepackage{float}
\usepackage{xcolor}
\usepackage[table]{xcolor} 
\usepackage{amsmath,amssymb,amsthm}
\usepackage{soul}

\title{On the achromatic index of Johnson graphs $J(n,2)$ 
}

\date{\today}

\newtheorem{theorem}{Theorem}[section]
\newtheorem{definition}[theorem]{Definition}

\newtheorem{lemma}[theorem]{Lemma}
\newtheorem{corollary}[theorem]{Corollary}
\newtheorem{proposition}[theorem]{Proposition}

\usepackage{amsfonts}

\newcommand{\Z}{\mathbb{Z}}

\newcommand{\blue}[1]{\textcolor{blue}{#1}}
\newcommand{\red}[1]{\textcolor{red}{#1}}

\definecolor{greeen}{RGB}{47, 141, 26}

\begin{document}

\author{Gabriela Araujo-Pardo\thanks{Instituto de Matem\'aticas, Universidad Nacional Aut\'onoma de M\'exico, Campus Juriquilla, Mexico, {\tt garaujo@im.unam.mx}, supported by DGAPA-M\'exico PAPIIT IN113324, SECHITI-M\'exico: CBF2023-2024-552 and a grant from the Research Vice-rectorate of Universitat de Lleida.}
    \and
	Cristina Dalf\'o\thanks{Departament de Matem\`atica, Universitat de Lleida, Igualada (Barcelona), {\tt cristina.dalfo@udl.cat}, supported by AGAUR from the Catalan Government under project 2021SGR00434, MICINN from the Spanish Government under project PID2025-171721NB-I00 and DGAPA-M\'exico PAPIIT IN113324.}
    \and
    M\'onica A. Reyes\thanks{Departament de Matem\`atica, Universitat de Lleida, Igualada (Barcelona), {\tt monicaandrea.reyes@udl.cat}, supported by AGAUR from the Catalan Government under project 2021SGR00434 and FI SDUR 2022, MICINN from the Spanish Government under project PID2025-171721NB-I00, and SECHITI-M\'exico: CBF2023-2024-552.}
    }

\maketitle
\begin{abstract}
In this paper, we study proper and complete edge-colorings of Johnson graphs $J(n,2)$, also called $n$-triangular graphs. They are isomorphic both to the 2-token graphs of complete graphs and to the line graphs of complete graphs. A $t$-edge-coloring of a graph $G$ is a function that assigns one color from $\{1,2,\ldots,t\}$ to each edge. Such a coloring is called proper if no two incident edges receive the same color, and complete if every pair of distinct colors appears on a pair of incident edges. The achromatic index, denoted by $\alpha_2(G)$, is the largest integer $t$ for which $G$ admits a proper and complete $t$-edge-coloring. We establish new lower and upper bounds for $\alpha_2(J(n,2))$, provide explicit proper and complete edge-colorings attaining the lower bounds, and determine the exact value of $\alpha_2(J(n,2))$ for several values of $n$.
\end{abstract}



\section{Introduction and preliminaries}

Coloring is a central topic in graph theory, with applications ranging from scheduling and frequency assignment to network design and optimization. While the standard 
chromatic number measures the minimum number of colors required for a proper coloring, other parameters explore different extremal aspects of colorings. One of these parameters is the \emph{achromatic number}, introduced by Harary, Hedetniemi, and Prins \cite{hhp67} in 1967, which has attracted considerable attention in graph theory. It is the largest number of colors in a vertex-coloring that can be used while keeping the coloring both proper (no two adjacent vertices share the same color) and complete (any pair appears on at least one pair of adjacent vertices). Finding the achromatic number is known to be an NP-complete problem, even for simple graph families such as bipartite graphs, cographs, interval graphs, and trees; see 
Bodlaender \cite{b89}, and Farber, Hahn, Hell, and Miller  \cite{fhhm86}.
Further results on the achromatic number and related complete-coloring
parameters can be found in
Araujo-Pardo, Montellano-Ballesteros, Rubio-Montiel, and Strausz \cite{apms18}, 
Araujo-Pardo, Montellano-Balles\-teros, and Strausz \cite{apms11},
Araujo-Pardo, Montellano-Ballesteros, Strausz, and Rubio-Montiel \cite{apmsr14}, 
Araujo-Pardo and Rubio-Montiel \cite{apr18}, 
Cairnie and Edwards \cite{ce97}, 
Chartrand and Zhang \cite{cz09}, 
Hell and Miller\cite{hm76}, and 
Yegnanarayanan \cite{y01}.

The concept of complete coloring can also be applied to edge-colorings. The maximum integer $t$ for which a graph $G$ admits a proper and complete $t$-edge-coloring is called the \emph{achromatic index}, denoted by $\alpha_2(G)$. This parameter has been studied by several authors. For instance, the problem of determining the achromatic index for complete graphs was first considered by Bouchet~\cite{b78}, who showed a strong connection between this parameter and the existence of finite projective planes. Further results and bounds for the achromatic index in different graph classes can be found in Araujo-Pardo, Montellano-Ballesteros, Olsen, and Rubio-Montiel \cite{amor21}, and Chiang and Fu \cite{cf95}.

In this paper, we study the achromatic index of Johnson graphs $J(n,2)$, also known as the $n$-triangular graphs. More generally, the Johnson graph $J(n,k)$ has as vertices the $k$-subsets of an $n$-element set, where two vertices are adjacent whenever they differ in exactly one element (equivalently, their intersection has size $k-1$). It has order $\binom{n}{k}$ and is $k(n-k)$-regular. In particular, $J(n,2)$ is a $2(n-2)$-regular graph on $\binom{n}{2}$ vertices. Johnson graphs form a fundamental family of distance-regular graphs and have been extensively studied because of their rich combinatorial structure and connections with design theory and coding theory.

Johnson graphs form a particular family of token graphs. Indeed, $J(n,k)$ is isomorphic to the $k$-token graph of the complete graph $K_n$ (see Fabila-Monroy, Flores-Pe\~{n}aloza, Huemer, Hurtado, Urrutia, and Wood \cite{ffhhuw12}, and Dalf\'o, Duque, Fabila-Monroy, Fiol, Huemer, Trujillo-Negrete, and Zaragoza Mart\'inez \cite{ddffhtz2021}). Let $G=(V,E)$ be a graph and let $1\le k\le |V|$. The \emph{$k$-token graph} of $G$, denoted by $F_k(G)$, has vertex set
$$
V(F_k(G))=\{A\subseteq V(G): |A|=k\},
$$
where two vertices $A$ and $B$ are adjacent if and only if
$$
A\triangle B=\{a,b\},
$$
with $a\in A$, $b\in B$, and $ab\in E(G)$.

When $k=2$, we will write $ij$ instead of $\{i,j\}$ whenever convenient. Accordingly, $(ij,il)$ denotes the edge joining the vertices $ij$ and $il$.

Throughout this paper, all graphs are finite, simple, and undirected. A graph $G = (V, E)$ consists of a set of vertices $V(G)$ and a set of edges $E(G)$. We write $|V|$ and $|E|$ for the number of vertices and edges in $G$, respectively. Two edges are said to be \emph{incident} if they share a common vertex.

A \emph{$t$-vertex-coloring} of $G$ is a mapping
$$
\varphi:V(G)\to\{1,2,\ldots,t\}.
$$
A proper and complete $t$-vertex-coloring is called \emph{achromatic}. The \emph{chromatic number}, denoted by $\chi(G)$, is the minimum number of colors in a proper vertex-coloring of $G$, and the \emph{achromatic number}, denoted by $\alpha(G)$, is the maximum number of colors in an achromatic vertex-coloring. Harary, Hedetniemi, and Prins~\cite{hhp67} proved that if $\chi(G)\le t\le\alpha(G)$, then $G$ admits a complete $t$-vertex-coloring.

Analogously, a \emph{$t$-edge-coloring} of $G$ is a mapping
$$
\varphi : E(G) \to \{1, 2, \dots, t\}.
$$
Such a mapping induces a partition of $E(G)$ into \emph{color classes} $E_i = \varphi^{-1}(i)$, for $i \in \{1,\dots,t\}$.
A proper and complete $t$-edge-coloring is called \emph{achromatic}. The \emph{achromatic index}, denoted by $\alpha_2(G)$, is the maximum integer $t$ for which $G$ admits an achromatic $t$-edge-coloring.

It is well known that
$$
\alpha_2(G)=\alpha(L(G)).
$$






For the Johnson graph $J(n,2)$, it is isomorphic to the line graph
$L(K_n)$.
Consequently,
$ \alpha_2(J(n,2)) = \alpha(L(J(n,2))) = \alpha(L^2(K_n)), $ so finding the achromatic index of $J(n,2)$ is the same as finding the 
achromatic number of the second iterated line graph of the complete graph. 
This connects our problem directly to the general study of achromatic 
colorings of line graphs and iterated line graphs.

In this paper, we establish new lower and upper bounds for the achromatic 
index of the Johnson graphs $J(n,2)$, together with exact values for small 
instances. In particular, our new lower-bound constructions give 
approximately $n^2$ colors, nearly twice the elementary lower bound of 
$\binom{n}{2}$.

The remainder of the paper is organized as follows. In Subsection~\ref{subsec:lowerbound}, we construct edge-colorings attaining the lower bounds on the achromatic index of Johnson graphs $J(n,2)$. In Subsection  \ref{subsec:upperbound}, we give the upper bounds, and, in Subsection \ref{subsec:theorems}, we prove the main theorems of this paper. Finally, Section~\ref{sec:conclusion} summarizes the results and discusses directions for future research, particularly from the perspective of Johnson graphs as token graphs of complete graphs.

\subsection{Fundamental results}

In this section, we collect some basic observations and known results about complete edge-colorings and token graphs that will be used throughout the paper.

Let $ G $ be a graph with $ n $ vertices and $ m $ edges. As noted by Reyes, Dalfó, and Fiol ~\cite{rdf24} in Section~5, every edge $ e = \{u, v\} \in E(G) $ gives rise to $ \binom{n-2}{k-1} $ edges in the $k$-token graph $ F_k(G) $. This happens because we obtain an edge in $ F_k(G) $ every time a token moves between $ u $ and $ v $, while the remaining $ k - 1 $ tokens are fixed in the $ n - 2 $ remaining vertices. In particular, as the Johnson graph $ J(n,k) \cong F_k(K_n) $, each of the $ \binom{n}{2} $ edges of $ K_n $ induces $ \binom{n-2}{k-1} $ edges in $ J(n,k) $. This structure allows for a natural decomposition of the edge set of $ F_k(G) $ based on the edges of $ G $.

The following proposition shows that this decomposition gives rise to a complete edge-coloring.

\begin{definition}
Let $G$ be a finite simple graph and let $F_k(G)$ be its $k$-token graph. 
For each edge $e=\{u,v\} \in E(G)$, define
$$
\mathcal{C}_e := \{\{A,B\} \in E(F_k(G)) : A \triangle B = e\},
$$
that is, the set of edges of $F_k(G)$ induced by $e$. Equivalently,
$$
\mathcal{C}_e = \Bigl\{ \{S, (S \setminus \{u\}) \cup \{v\}\} : S \subset V(G),\, |S| = k,\, u \in S,\, v \notin S \Bigr\}.
$$
We define an edge-coloring
$$
\varphi:E(\mathcal{F}_k(G))\longrightarrow E(G)
$$
by
$$
\varphi(\{A,B\})=A\triangle B.
$$
Equivalently, for every $e\in E(G)$,
$$
C_e=\varphi^{-1}(e).
$$
Thus, the edges of $G$ are used as color labels.

\end{definition}

\begin{proposition}
Let $G$ be a finite simple graph and let $F_k(G)$ be its $k$-token graph, where $2 \le k \le |V(G)| - 2$. Then $F_k(G)$ admits an achromatic edge-coloring with $|E(G)|$ colors.\end{proposition}

\begin{proof}
We prove that the coloring $\varphi$ defined above is both proper and complete.\\
\textbf{Properness.} Fix an edge $e=\{u,v\}\in E(G)$. 
Consider two distinct edges in $\mathcal{C}_e$, say $\{A,A'\}$ and $\{B,B'\}$. 
These do not share any endpoint in $F_k(G)$, so the edges of $\mathcal{C}_e$ form a matching. 
Hence, $\varphi$ is a proper edge-coloring. \\
\textbf{Completeness.} 
Let $e=\{a,b\}$ and $f=\{c,d\}$ be two distinct edges of $G$. 
We show that there exists a vertex of $F_k(G)$ incident to one edge of $\mathcal{C}_e$ and one edge of $\mathcal{C}_f$. 
There are two cases to consider:\\
\emph{Case (i):} $e$ and $f$ are disjoint, i.e.\ $\{a,b\} \cap \{c,d\} = \emptyset$.  
Take a set $S \subset V(G)$ with $|S|=k$ such that $a,d \in S$, $b,c \notin S$.  
Such a set exists because, after fixing $a$ and $d$, the remaining
$k-2$ elements of $S$ can be chosen from
$V(G)\setminus\{a,b,c,d\}$, and
$0\leq k-2\leq |V(G)|-4$.
Then $\{S,\, (S\setminus\{a\}) \cup \{b\}\} \in \mathcal{C}_e, 
\quad 
\{S,\, (S\setminus\{d\}) \cup \{c\}\} \in \mathcal{C}_f,$
and both edges share the vertex $S$, hence they are adjacent in $F_k(G)$. \\
\emph{Case (ii):} $e$ and $f$ share one vertex, say $a=c$.  
Take a set $S \subset V(G)$ with $|S|=k$ such that $a \in S$, $b,d \notin S$.  
Such a set exists because, after fixing $a$, the remaining $k-1$
elements can be chosen from $V(G)\setminus\{a,b,d\}$, and
$0\leq k-1\leq |V(G)|-3$.
Then
$$
\{S,\, (S\setminus\{a\}) \cup \{b\}\} \in \mathcal{C}_e, 
\quad 
\{S,\, (S\setminus\{a\}) \cup \{d\}\} \in \mathcal{C}_f,
$$
and again both edges share the vertex $S$, so they are adjacent.
In both cases, there is a vertex of $F_k(G)$ incident to one edge colored $e$ and one edge colored $f$. 
Therefore, $\varphi$ is a complete edge-coloring of $F_k(G)$.
\end{proof}

As a consequence of the previous proposition, we obtain the following general lower bound on the achromatic index of token graphs.

\begin{corollary}
Let $G$ be a simple graph and let $F_k(G)$ be its $k$-token graph, where $2 \le k \le |V(G)| - 2$. Then the achromatic index of its $k$-token graph satisfies
$$
\alpha_2(F_k(G)) \geq |E(G)|.
$$
\end{corollary}

In particular, when $ G = K_n $, we have $ F_k(G) = J(n,k) $ and $ |E(K_n)| = \binom{n}{2} $. Therefore, we obtain the following immediate lower bound for Johnson graphs:

\begin{corollary}
For all integers $ n \geq 2 $ and $ 2 \leq k \leq n-2 $,
$$
\alpha_2(J(n,k)) \geq \binom{n}{2}.
$$
\end{corollary}


\section{The achromatic index of Johnson graphs}
\label{sect:results}

In this section, we present the main results of the paper. We begin with the following theorem.

\begin{theorem}\label{lower}
For every integer $n\geq 3$, the achromatic index of the Johnson graph
$J(n,2)$ satisfies
$$
\alpha_2(J(n,2))\geq
\begin{cases}
n(n-2), 
    & \text{if $n$ is odd},\\[1mm]
7, 
    & \text{if $n=4$},\\[1mm]
n(n-3)+3, 
    & \text{if $n\geq 6$ and $n\equiv 0\pmod{6}$},\\[1mm]
n(n-3)+2, 
    & \text{if $n\geq 6$ and $n\not\equiv 0\pmod{6}$}.
\end{cases}
$$
\end{theorem}

        
    


To state the upper bound, we introduce the following notation. For $n\ge3$, let
$f,g:\mathbb{Z}_{0}\to\mathbb{Z}_{\blue{\geq0}}$ be defined by
$$
f(x)=\left\lfloor\frac{\binom{n}{2}(n-2)}{x}\right\rfloor,
\qquad
g(x)=4nx-10x+1.
$$
Furthermore, let
$$
x_0=x_0(n):=-
\frac{1}{8n-20}
+\frac12\sqrt{
\frac{n^2}{2}
-\frac n4
+\frac38
+\frac{\frac{15n}{4}-\frac{73}{8}}{4n^2-20n+25}},
$$
and define
$$
U(n):=
\max\left\{
\min\{f(\lfloor x_0\rfloor),g(\lfloor x_0\rfloor)\},
\min\{f(\lceil x_0\rceil),g(\lceil x_0\rceil)\}
\right\}.
$$
\begin{theorem}\label{teoupperbounds}
For any integer $n\geq 3$, we have 
$$\alpha_2(J(n,2)) \leq  U(n).$$

\end{theorem}


Table~\ref{tab:betterupper-bounds} lists the values of the upper bound $U(n)$ for several values of $n$.

Combining Theorems~\ref{lower} and~\ref{teoupperbounds}, we obtain the following corollary, which determines the exact value of $\alpha_2(J(n,2))$ for $n=3,4,5$ and provides lower and upper bounds for all $n\ge6$.

\begin{corollary}\label{mainth}
The achromatic index of the Johnson graph $J(n,2)$ satisfies
$$
\alpha_2(J(3,2))=3,\qquad
\alpha_2(J(4,2))=7,\qquad
\alpha_2(J(5,2))=15.
$$
Moreover, if $n\geq 7$ is odd, then
$$
n(n-2)
=
f\left(\frac{n-1}{2}\right)
\leq
\alpha_2(J(n,2))
\leq
U(n).
$$
If $n\geq 6$ is even and $n\not\equiv 0\pmod{6}$, then
$$
n(n-3)+2
=
f\left(\frac{n}{2}\right)
\leq
\alpha_2(J(n,2))
\leq
U(n).
$$
Finally, if $n\geq 6$ and $n\equiv 0\pmod{6}$, then
$$
n(n-3)+3
=
f\left(\frac{n}{2}\right)+1
\leq
\alpha_2(J(n,2))
\leq
U(n).
$$

\end{corollary}

\subsection{A construction that attains the lower bound on $\alpha_2(J(n,2))$}
\label{subsec:lowerbound}

In this subsection, we focus on obtaining a lower bound for the achromatic index of the Johnson graph $J(n,2)$.
We treat odd $n\geq3$ and even $n\geq6$. The case $n=4$ is handled separately.
We construct an 
achromatic edge-coloring of $ J(n,2) $ by defining color classes of two different types, depending on whether $n$ is even or odd. We begin with the case of odd $n$.

\subsubsection{The case of odd $n$}

In this subsection, we define an achromatic $t$-edge-coloring of $J(n,2)$ for odd $n$, with $t=n(n-2)$. The construction is based on an edge decomposition of $J(n,2)$ into $n$ subgraphs, called \emph{boxes}. Each box is edge-colored independently using a $1$-factorization,  while the union of the boxes guarantees completeness.
\paragraph{Construction.}

For each $i\in \mathbb{Z}_n$, let $\mathcal{B}_i$ be the subgraph of $J(n,2)$ that we call the {\emph{box of $i$}}, where
$$V(\mathcal{B}_i)=\{ij\ | \text{ for}\ j\in \Z_n\text{, and}\ j\not=i\},\ $$ and
$$
E(\mathcal{B}_i)=\{(ij,il)\ :\ j,l \in \mathbb{Z}_n,\ j\neq l,\ j\neq i,\ l\neq i\}.
$$
Notice that $\mathcal{B}_i$ is isomorphic to a complete graph of even order  $n-1$, and it is well known (see Chartrand, Lesniak, and Zhang \cite{CLZ16}) that it admits a 1-factorization with $(n-2)$ factors.

Let $\varphi$ be the edge-coloring of $J(n,2)$ such that each factor of $\mathcal{B}_i$ is assigned a distinct color; then we have a 1-factorization of $\mathcal{B}_i$ with $(n-2)$ colors. Observe that the restriction of $\varphi$ to $\mathcal{B}_i$ gives a proper complete $(n-2)$-edge-coloring of $\mathcal{B}_i$, where every vertex is incident to an edge of each color.

Moreover, $\mathcal{B}_i$ is a subgraph of $J(n,2)$ with the property that each color appears in an edge incident to a given vertex, for any vertex. For simplicity, in this paper, we say that the box $\mathcal{B}_i$ is the \emph{owner} of its colors (see Figure~\ref{fig:box-example}). 

\begin{figure}[!ht]
	\begin{center}
\includegraphics[width=12cm, page=11]{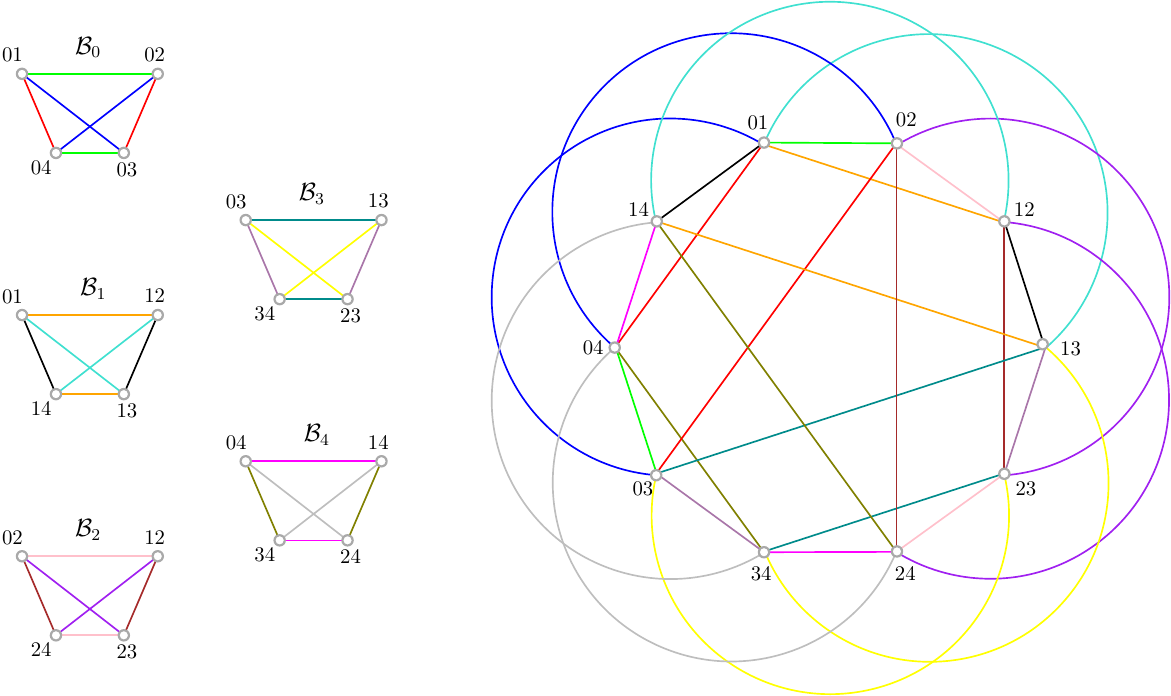}
  \vskip-.3cm
	\caption{The odd case $n=5$. 
Left: The subgraphs $\mathcal{B}_0,\mathcal{B}_1,\mathcal{B}_2,\mathcal{B}_3$, and $\mathcal{B}_4$ of $J(5,2)$ with the restriction of $\varphi$. Right: The achromatic 15-edge-coloring $\varphi$ of $J(5,2)$.}
		\label{fig:box-example}
	\end{center}
\end{figure}
 
We do the same for each $i\in \Z_n$ so that $\varphi$ is an $n(n-2)$-coloring of $J(n,2)$. Notice that for distinct $i$ and $j$, the vertex sets $V(\mathcal{B}_i)$ and $V(\mathcal{B}_j)$ always intersect (that is, $V(\mathcal{B}_i)\cap V(\mathcal{B}_j)=\{ij\}$), while the edge sets $\{E(\mathcal{B}_i)\}_{i\in \mathbb{Z}_n}$ form a partition of $E(J(n,2))$ and, hence, are pairwise disjoint and cover all edges of $J(n,2)$  
(see Figure \ref{fig:box-example}).

\paragraph{Properness.} The coloring $\varphi$ of $J(n,2)$ is proper since each vertex $ij$ is incident only to edges from $\mathcal{B}_i$ and $\mathcal{B}_j$. Since $\mathcal{B}_i$ and $\mathcal{B}_j$ use disjoint color sets and each induces a proper edge-coloring, all edges incident to $ij$ receive distinct colors.

\paragraph{Completeness.}
To prove completeness, let $c_1$ and $c_2$ be two distinct colors. 

If both colors belong to the same box, then they are incident at every vertex of that box, since the box owns its colors. 

Otherwise, suppose that $c_1$ belongs to $B_i$ and $c_2$ belongs to $B_j$, with $i \neq j$. Since $V(B_i)\cap V(B_j)=\{ij\},$ and each box owns its colors, the two colors are incident at the vertex $ij$. Therefore, the coloring is complete. See Figure \red{\ref{fig:box-example}} for an example.

\subsubsection{The case of even $n$}
We first construct an achromatic edge-coloring $\varphi$ of $J(n,2)$ with $n(n-3)+2$ colors. We then modify this construction to obtain an achromatic edge-coloring $\varphi_0$ with $n(n-3)+3$ colors whenever
$n \equiv 0 \pmod{6}$.

\paragraph{Construction of $\varphi$.}

Since $n$ is even, we may choose, without loss of generality, the
following $1$-factor of $K_n$:
$$
F=
\bigl\{
\{0,1\},\{2,3\},\ldots,\{n-2,n-1\}
\bigr\}.
$$
Let
$$
p:\mathbb{Z}_n\longrightarrow\mathbb{Z}_n
$$
be the fixed-point-free involution induced by $F$, defined by
$$
p(i)=j
\quad\Longleftrightarrow\quad
\{i,j\}\in F.
$$
Thus,
$$
p(p(i))=i
\qquad\text{and}\qquad
p(i)\neq i
$$
for every $i\in\mathbb{Z}_n$. For the chosen $1$-factor $F$, we have
$$
p(i)=
\begin{cases}
i+1, & \text{if $i$ is even},\\
i-1, & \text{if $i$ is odd}.
\end{cases}
$$
Let
$$
I=\{0,2,\ldots,n-2\}.
$$
so that
$$
F=\bigl\{\{i,p(i)\}:i\in I\bigr\}.
$$
Since $p(i)=i+1$ for every $i\in I$, we shall continue to write
$B_{i+1}$ for $B_{p(i)}$ and $DB_{i(i+1)}$ for the corresponding
double-box.

Recall that all indices are taken modulo $n$. 
For each $i\in I$, we define two subgraphs of $J(n,2)$, denoted by
$B_i$ and $B_{p(i)}$, called the box of $i$ and the box of $p(i)$,
respectively, where
$$
V(B_i)
=
\left\{
ij:j\in\mathbb{Z}_n\setminus\{i,p(i)\}
\right\},
$$
$$
V(B_{p(i)})
=
\left\{
p(i)j:j\in\mathbb{Z}_n\setminus\{i,p(i)\}
\right\},
$$
and
$$
E(B_i)
=
\left\{
(ij,ik):
j,k\in\mathbb{Z}_n\setminus\{i,p(i)\},
\ j\neq k
\right\},
$$
$$
E(B_{p(i)})
=
\left\{
\bigl(p(i)j,p(i)k\bigr):
j,k\in\mathbb{Z}_n\setminus\{i,p(i)\},
\ j\neq k
\right\}.
$$

 

Note that the definition of a box in the odd case differs from that in the even case. Each box is isomorphic to $K_{n-2}$ and, as in the previous case, they admit a 1-factorization with $(n-3)$ factors. 

For each $r\in\mathbb{Z}_n$, let $\mathcal{P}_r$ be a set of
$n-3$ colors, and choose these palettes to be pairwise disjoint; that
is,
$$
\mathcal{P}_r\cap\mathcal{P}_s=\emptyset
\qquad\text{whenever }r\neq s.
$$
We color the $n-3$ factors of the $1$-factorization of $B_r$
bijectively with the colors of $\mathcal{P}_r$. Thus, each color
belongs to a unique box (see Figure \ref{fig:correspondence2}).


Thus, for each $((i)(i+1))$, we use $2(n-3)$ colors in total because the two subgraphs are assigned different colors. 

\begin{figure}[!ht]
	\begin{center}
\includegraphics[width=12cm,page=3]{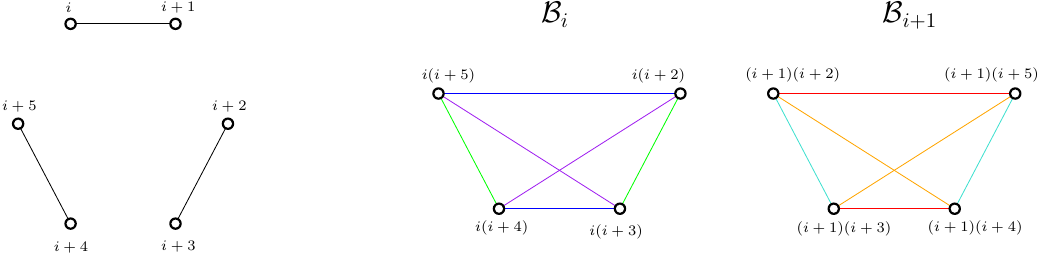}
  \vskip-.3cm
	\caption{The even case \(n=6\).
Left: A 1-factor \(\mathcal{F}\) of the complete graph \(K_6\).
Right: The subgraphs \(\mathcal{B}_i\) and \(\mathcal{B}_{i+1}\) of $J(6,2)$.}
		\label{fig:correspondence2}
	\end{center}
\end{figure}

Consequently, the restriction of $\varphi$ to each box is a proper complete edge-coloring. In particular, every box owns its colors.

The boxes alone are not sufficient to guarantee incidences between colors from different boxes. Therefore, we combine pairs of boxes into larger structures, called \emph{double-boxes}.

For each $i\in I$, we join $B_i$ and $B_{i+1}$
to form a new subgraph of $J(n,2)$, which we call the {\emph{double-box of $i(i+1)$}}, denoted by $\mathcal{DB}_{i(i+1)}$, as follows. 
Fix the vertex $i(i+1)$, which is adjacent to every vertex of both
$\mathcal B_i$ and $\mathcal B_{i+1}$.
Since each box has order $n-2$, the vertex $i(i+1)$ is incident with
exactly $n-2$ edges joining $\mathcal B_i$ and another $n-2$ joining
$\mathcal B_{i+1}$.\\
We first assign two new colors \(c\) and \(d\) to two specific edges incident to \(i(i+1)\):
the edge joining \(i(i+1)\) with \(i(i+2)\) receives color \(c\), and the edge joining
\(i(i+1)\) with \((i+1)(i+3)\) receives color \(d\), where all numbers are taken modulo \(n\).\\

For the remaining edges incident to $i(i+1)$, we choose a bijective 
assignment of the $n-3$ colors assigned within $B_{i+1}$ to the $n-3$ 
remaining edges joining $i(i+1)$ to vertices in $B_i$. Similarly, we 
choose a bijective assignment of the $n-3$ colors assigned within $B_i$ 
to the remaining edges joining $i(i+1)$ to vertices in $B_{i+1}$. Hence 
every color from each box is used exactly once at the central vertex \(i(i+1)\).
We use the colors $c$ and $d$ in every double-box
$DB_{i(i+1)}$, for $i\in I$, according to the preceding rule.

\begin{figure}[!ht]
	\begin{center}
\includegraphics[width=13cm,page=5]{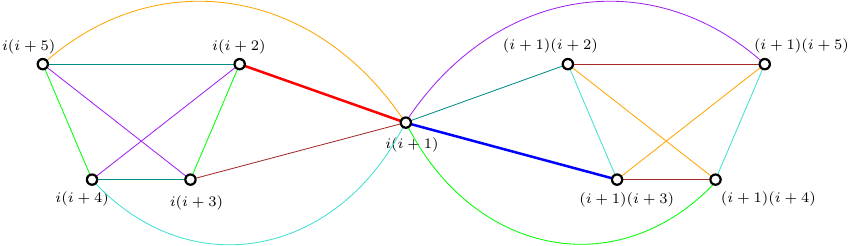}
  \vskip-.3cm
	\caption{ The even case \(n=6\): The \emph{double-box of $i(i+1)$} \(\mathcal{DB}_{i(i+1)}\), where color \(c\) is represented in red and color \(d\) in blue.}
		\label{fig:correspondence3}
	\end{center}
\end{figure}

Repeating the previous construction for every edge of the
1-factor $\mathcal F$ colors every edge of $J(n,2)$.
Indeed, every edge of $J(n,2)$ has a unique common label and may
therefore be written as
$$
(ij,ik),
$$
where $i,j,k\in\mathbb{Z}_n$ are pairwise distinct. If
$j,k\notin\{i,p(i)\},$
then $(ij,ik)$ belongs to the unique box $B_i$. Otherwise, exactly one
of $j$ and $k$ is equal to $p(i)$ and, hence, $(ij,ik)$ is incident
with the central vertex $i\,p(i)$. Such an edge is colored in the
corresponding double-box $DB_{i\,p(i)}$. Consequently, every edge of
$J(n,2)$ is colored exactly once.


Therefore, the number of colors of $\varphi$ is $t=n(n-3)+2=n^2-3n+2$. 


\paragraph{Properness.}
We now prove that $\varphi$ is a proper edge-coloring of $J(n,2)$.

First, consider a double-box $\mathcal{DB}_{i(i+1)}$. The restrictions of
$\varphi$ to the boxes $\mathcal{B}_i$ and $\mathcal{B}_{i+1}$ are proper
edge-colorings induced by their respective $1$-factorizations. Moreover,
when constructing the double-box, the colors of one box are reused only on
the edges joining the central vertex $i(i+1)$ to the other box. Hence,
every vertex of $\mathcal{DB}_{i(i+1)}$ is incident only with edges of
distinct colors.

It remains to verify properness at vertices that belong to boxes from
two distinct double-boxes. 
The color $c$ appears
precisely on the edges
$$
\mathcal{C}_c
=
\left\{
\bigl(i(i+1),i(i+2)\bigr):i\in I
\right\},
$$
whereas the color $d$ appears precisely on the edges
$$
\mathcal{C}_d
=
\left\{
\bigl(i(i+1),(i+1)(i+3)\bigr):i\in I
\right\}.
$$
Recall that all indices are taken modulo $n$. 
Both $\mathcal{C}_c$ and $\mathcal{C}_d$ are matchings, since the
edges in each of these two sets are pairwise nonincident.

Let $ij$ be a noncentral vertex; equivalently, $j\neq p(i).$
In this case, the four indices
$$
i,\quad j,\quad p(i),\quad p(j)
$$
are pairwise distinct.
The edges incident with $ij$ are of the following four types:
\begin{enumerate}
    \item edges of $B_i$ incident with $ij$, whose colors belong to
    $\mathcal{P}_i$;
    \item edges of $B_j$ incident with $ij$, whose colors belong to
    $\mathcal{P}_j$;
    \item the edge $\bigl(ij,i\,p(i)\bigr),$
    whose color belongs to $\mathcal{P}_{p(i)}$, unless it is one of
    the special edges;
    \item the edge $\bigl(ij,j\,p(j)\bigr)$,
    whose color belongs to $\mathcal{P}_{p(j)}$, unless it is one of
    the special edges.
\end{enumerate}
Since
$$
\mathcal{P}_i,\quad
\mathcal{P}_j,\quad
\mathcal{P}_{p(i)},\quad
\mathcal{P}_{p(j)}
$$
are pairwise disjoint, and the restriction of the coloring to each box
is proper; no two nonspecial edges incident with $ij$ receive the same
color. Moreover, each special color class is a matching, so no two
edges of the same special color are incident with $ij$.
Finally, at every central vertex $i\,p(i)$, the bijective assignments
used in the construction ensure that each color in
$$
\mathcal{P}_i\cup\mathcal{P}_{p(i)}
$$
appears exactly once, together with the two distinct special colors
$c$ and $d$. Therefore, all the edges incident with a central vertex
also receive distinct colors. Hence, $\varphi$ is a proper
edge-coloring of $J(n,2)$.


\paragraph{Completeness.} We now prove that $\varphi$ is complete.

First, consider two colors belonging to the same double-box
$\mathcal{DB}_{i(i+1)}$. If both belong to the same box
$\mathcal{B}_i$, then they meet in $\mathcal{B}_i$, since
$\mathcal{B}_i$ owns its colors. If one belongs to $\mathcal{B}_i$
and the other to $\mathcal{B}_{i+1}$, then they meet at the vertex
$i(i+1)$. Moreover, the additional colors $c$ and $d$ are incident
with every other color of $\mathcal{DB}_{i(i+1)}$ at this vertex.

Finally, let one color belong to
$\mathcal{DB}_{i(i+1)}$ and another to a different double-box
$\mathcal{DB}_{k(k+1)}$. Without loss of generality, suppose they
belong to $\mathcal{B}_i$ and $\mathcal{B}_{k+1}$, respectively.
Since the boxes $\mathcal{B}_i$ and $\mathcal{B}_{k+1}$ intersect
exactly in the vertex $i(k+1)$, the two colors meet at that vertex.
Therefore, every pair of colors is incident, and $\varphi$ is a
complete edge-coloring of $J(n,2)$.



\paragraph{Improvement when $n\equiv0\pmod6$.} We now show that, when $n \equiv 0 \pmod{6}$, the previous construction can be modified to obtain an achromatic edge-coloring $\varphi_0$ with one additional color. 

The idea is to replace the two extra colors $c$ and $d$ by three colors $c, d, e$, distributed among triples of consecutive \emph{double-boxes} in such a way that every new color still meets every other color. Specifically, for each $i \in \{0, \dots, n-1\}$ with $i \equiv 0 \pmod{6}$, these three colors are assigned to the extra edges of the consecutive double-boxes as follows:
\begin{align*}
\mathcal{DB}_{i(i+1)} &\text{ receives the colors } \{c, d\}, \\
\mathcal{DB}_{(i+2)(i+3)} &\text{ receives the colors } \{c, e\}, \text{ and} \\
\mathcal{DB}_{(i+4)(i+5)} &\text{ receives the colors } \{d, e\}.
\end{align*}

Since each new color appears in only two of the three double-boxes, it is not immediate that it meets all the colors contained in the remaining double-box. To do this, we choose the two special edges for each color to have one endpoint in the missing double-box. Consequently, every new color is adjacent to all the colors owned by the
corresponding boxes. More precisely, for every
$i\equiv0\pmod6$,
we define
\begin{itemize}
\item $\varphi_0(i(i+1),i(i+4))=\varphi_0((i+2)(i+3),(i+2)(i+5))=c$
\item $\varphi_0(i(i+1),(i+1)(i+2))=\varphi_0((i+4)(i+5),(i+4)(i+3))=d$
\item $\varphi_0((i+2)(i+3),(i+3)i)=\varphi_0((i+4)(i+5),(i+5)(i+1))=e$
\end{itemize}

\begin{figure}[t]
	\begin{center}
\includegraphics[width=13cm,page=10]{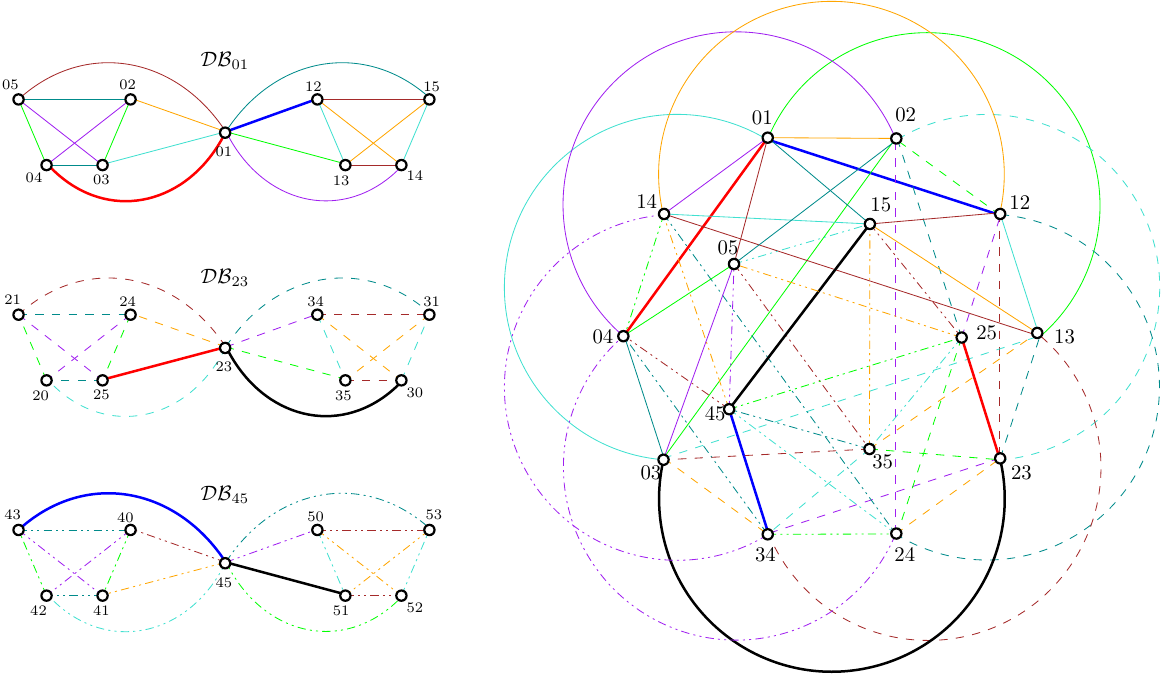}
  \vskip-.3cm
	\caption{ The even case $n=6$: Left: Double-boxes \(\mathcal{DB}_{01}\), \(\mathcal{DB}_{23}\), \(\mathcal{DB}_{45}\) (red: color \(c\), blue: color \(d\). black: color $e$). Right: The corresponding achromatic 21-edge-coloring in \(J(6,2)\).}
 		\label{fig:j(6,2)21}
	\end{center}
\end{figure}
For example, when $n=6$, the colors are distributed as follows:
\begin{itemize}
    \item $\mathcal{DB}_{01}$ receives $\{c,d\}$, with $\varphi_0(01,04)=c$ and $\varphi_0(01,12)=d$;
    \item $\mathcal{DB}_{23}$ receives $\{c,e\}$, with $\varphi_0(23,25)=c$ and $\varphi_0(23,30)=e$;
    \item $\mathcal{DB}_{45}$ receives $\{d,e\}$, with $\varphi_0(45,43)=d$ and $\varphi_0(45,51)=e$.
\end{itemize}
This yields a $21$-complete coloring $\varphi_0$ of $J(6,2)$, as shown in Figure~\ref{fig:j(6,2)21}.

\paragraph{Properness.}
The coloring $\varphi_0$ remains proper because, for each special color 
$q \in \{c,d,e\}$, the edges colored $q$ are pairwise nonincident. 
This follows from the construction within each block of six consecutive 
indices, together with the fact that different blocks involve disjoint 
sets of indices. Hence no vertex is incident with two edges of the 
same color.

\paragraph{Completeness.}

The previous argument for completeness still applies to the original
colors.
The only additional point is to verify that each of the new colors
$c,d,e$
meets all the colors of the double-box in which it does not appear.
This follows directly from the special choice of the edges listed above.
Hence every pair of colors is incident, and
$\varphi_0$
is a complete edge-coloring.

\subsection{Upper Bounds for $\alpha_2(J(n,\blue{2}))$}
\label{subsec:upperbound}

The techniques used in this section were introduced by Jamison \cite{j89}, Araujo-Pardo, Montellano-Ballesteros, and Strausz \cite{apms11}, and Araujo-Pardo, Montellano-Ballesteros, Strausz, and Rubio-Montiel \cite{apmsr14}, and were used in almost all the papers cited in this work. 

We now optimize the preceding estimate and show that only two integer values of $x$ need to be considered.


Let $ \varphi : E(J(n,2)) \to [t] $ be an achromatic edge-coloring of the Johnson graph $J(n,2)$, using $t$ colors. 

Let $x$ denote the size of the smallest color class; that is, $x = \min_{i \in [t]} |\varphi^{-1}(i)|$. Each color class must form a matching (recall that it is an independent set in the line graph of $J(n,2)$). We consider two constraints:

\begin{itemize}
    \item[(i)] Since the total number of edges in $J(n,2)$ is $\binom{n}{2}(n - 2)$, the number of color classes is at most
    $$
    f(x) = \left\lfloor \frac{\binom{n}{2} \cdot (n - 2)}{x} \right\rfloor.
    $$
    
\item[(ii)] On the other hand, to ensure completeness, each color must appear together with every other color on some vertex. If a color class contains $x$ independent edges, then each edge can be incident to at most 2($2n -5$)=4n-10
other edges. This bound follows from the fact that in $J(n,k)$ every vertex has degree $k(n-k)$. So, 
this yields an upper bound on the total number of colors in the achromatic edge-coloring:
$$
g(x) = 4nx - 10x + 1, 
$$
where we added 1 to count the color itself.
\end{itemize}

In summary, for a fixed $x$, the number of colors $t$ is bounded by both $f(x)$ (due to the total number of edges) and $g(x)$ (due to completeness). Hence, we obtain the following general upper bound:
\begin{equation}
\label{equuperbound}
 \alpha_2(J(n,2)) \leq \max_{x \in \mathbb{N}} \left\{ \min\left( f(x), g(x) \right) \right\}.   
\end{equation}

To improve the previous bound, we analyze the behavior of both functions. 
Since $f(x)$ is the floor of a decreasing function and $g(x)$ is increasing, 
the maximum value of $\min\{f(x),g(x)\}$ is attained near the point where 
the two curves intersect. To formalize this, we consider the continuous 
version of $f$, defined as:
$$
F:\mathbb R_{>0}\to\mathbb R,\qquad
F(x)=\frac{\binom n2(n-2)}{x}.
$$
Since $F$ is strictly decreasing and $g$ is strictly increasing, there exists a unique real number $x_0>0$ such that
$F(x_0)=g(x_0)$.
\begin{lemma}\label{lem:op}
Let $x_0$ be the unique positive solution of
$F(x)=g(x)$.
Then
$$
\max_{x\in\mathbb N}\min\{f(x),g(x)\}
=
\max\left\{
\min\{f(\lfloor x_0\rfloor),g(\lfloor x_0\rfloor)\},
\min\{f(\lceil x_0\rceil),g(\lceil x_0\rceil)\}
\right\}.
$$
\end{lemma}
\begin{proof}
Since $F$ is strictly decreasing and $g$ is strictly increasing, we have
$F(x)\ge g(x)$ for every $x\le x_0$, and
$F(x)<g(x)$ for every $x>x_0$.\\
Since $g$ takes integer values on $\mathbb N$, it follows that
$f(x)=\lfloor F(x)\rfloor\ge g(x)$ whenever $x\le x_0$, whereas
$f(x)<g(x)$ whenever $x>x_0$. Hence,
$$
\min\{f(x),g(x)\}=
\begin{cases}
g(x), & x\le x_0,\\
f(x), & x>x_0.
\end{cases}
$$
Therefore, $\min\{f(x),g(x)\}$ is increasing on
$\{x\in\mathbb N:x\le x_0\}$ and non-increasing on
$\{x\in\mathbb N:x>x_0\}$. Consequently, its maximum is attained at one of the two integers nearest to $x_0$, namely, $\lfloor x_0\rfloor$ or $\lceil x_0\rceil$.
\end{proof}

We now compute the value of $x_0$. The equation $F(x) = g(x)$ can be written explicitly as
$$
\frac{\binom{n}{2}(n-2)}{x} = (4n-10)x+1,
$$
giving the quadratic equation
$$
(4n-10)x^2+x-\binom{n}{2}(n-2)=0.
$$
The unique positive solution $x_0$ to this equation is given by
\begin{equation}
\label{x0}
x_0 = \frac{-1+\sqrt{1+4(4n-10)\binom{n}{2}(n-2)}}{8n-20},
\end{equation}
which simplifies to
\begin{equation}
\label{x0_simplified}
x_0 =- \frac{1}{8n-20} + \frac12 \sqrt{\frac{n^2}{2} -\frac n4 +\frac38 + \frac{\frac{15n}{4}-\frac{73}{8}}{4n^2-20n+25}}.
\end{equation}
In particular, note that $x_0 \approx \frac{n}{\sqrt{8}}$. The resulting upper bounds for $4 \leq n \leq 12$ are summarized in Table~\ref{tab:betterupper-bounds}.  This calculation of the 
intersection point $x_0$, together with Lemma~\ref{lem:op}, gives us 
the desired upper bound and completes the proof of Theorem~\ref{teoupperbounds}.

Table \ref{tab:betterupper-bounds} gives us the values that appear as a consequence of Theorem \ref{teoupperbounds}.

\begin{table}[t]
\centering
\begin{tabular}{|c|c|c|c|c|c|}
\hline
$n$ & $x$ & $f(x)$ & $g(x)$ & $\min(f,g)$ & Upper bound \\
\hline
4 & 1 & 12 & 7 & 7 & 7\\
& 2 & 6 & 13 & 6 &  \\
\hline
5 & 1 & 30 & 11 & 11 & 15 \\
& 2 & 15 & 21 & 15 & \\
\hline
6 & 2 & 30 & 29 & 29 & 29 \\
& 3 & 20 & 43 & 20 & \\
\hline
7 & 2 & 52  & 37  & 37   &  37 \\
& 3 & 35  & 55  & 35   &  \\
\hline
8 & 2 & 84  & 45   & 45   & 56 \\
 & 3 & 56  & 67   & 56   &  \\
\hline 
9 & 3 & 84 & 79 & 79 & 79 \\
& 4 & 63 & 105 &63 & \\
\hline
10 & 3 & 120 & 91 & 91 & 91 \\
& 4 & 90 & 121 &90 & \\
\hline
11 & 3 & 165 & 103 & 103 & 123 \\
& 4 & 123 & 137 & 123 &  \\
\hline
12  & 4 & 165 & 153 & 153 & 153 \\
& 5 & 132 & 191 & 132 &  \\
\hline
\end{tabular}
\caption{Upper bounds on $\alpha_2(J(n,2))$ using the results of Theorem \ref{teoupperbounds}. 
}
\label{tab:betterupper-bounds}
\end{table}


\subsection{Proof of the Main Corollary}
\label{subsec:theorems}
In this section, we complete the proof of the main results by
establishing Corollary~\ref{mainth}. The constructions in Subsection~\ref{subsec:lowerbound} prove
Theorem~\ref{lower} for $n\neq 4$, while the case $n=4$ is established
below. Theorem~\ref{teoupperbounds} was proved in Subsection~\ref{subsec:upperbound}.

\paragraph{Proof of Corollary \ref{mainth}.}

We distinguish two cases.

\smallskip

\noindent
\textbf{Case 1.} $n$ odd.

If $n=3$, then $J(3,2)$ is isomorphic to $K_3$, and therefore
$\alpha_2(J(3,2))=3$.

For $n=5$, Subsection \ref{subsec:lowerbound} provides a $15$-achromatic
edge-coloring of $J(5,2)$, while Table \ref{tab:betterupper-bounds} shows that
the upper bound is also $15$. Hence,
$\alpha_2(J(5,2))=15$.

For odd integers $n\ge7$, Theorem \ref{lower} gives
$$
f\!\left(\frac{n-1}{2}\right)=n(n-2)\le \alpha_2(J(n,2)),
$$
while Theorem \ref{teoupperbounds} gives the corresponding upper bound
$\alpha_2(J(n,2))\le U(n)$.

\smallskip

\noindent
\textbf{Case 2.}  $n$ even.

We first consider the case $n=4$. Figure \ref{J(4,2)} shows an achromatic
$7$-edge-coloring of $J(4,2)$, and therefore
$\alpha_2(J(4,2))\ge7$. 
The coloring shown in Figure~\ref{fig:J(4,2)} can be stated explicitly as follows:
$$
\begin{aligned}
\mathcal{C}_1 &= \{(01,12),(23,03)\},\\
\mathcal{C}_2 &= \{(12,13),(02,23)\},\\
\mathcal{C}_3 &= \{(13,23),(03,02)\},\\
\mathcal{C}_4 &= \{(01,02)\},\\
\mathcal{C}_5 &= \{(01,13)\},\\
\mathcal{C}_6 &= \{(13,03),(12,02)\},\\
\mathcal{C}_7 &= \{(01,03),(23,12)\}.
\end{aligned}
$$
Each $\mathcal{C}_i$, for $i\in\{1,\ldots,7\}$, is a matching.
Moreover, for every pair of distinct indices $i,j\in\{1,\ldots,7\}$,
there exist edges $e_i\in\mathcal{C}_i$ and
$e_j\in\mathcal{C}_j$ that are incident. Therefore, these color
classes define a proper and complete $7$-edge-coloring of $J(4,2)$.

On the other hand, by
Theorem \ref{teoupperbounds},
$\alpha_2(J(4,2))\le7$. Consequently,
$\alpha_2(J(4,2))$ $=7$.

For even integers $n\ge6$, Theorem \ref{lower} gives the lower bound,
while Theorem \ref{teoupperbounds} provides the corresponding upper bound.

This completes the proof.

\medskip

\noindent
\textbf{Observation.} The lower bounds obtained in Theorem \ref{lower} are precisely the values of $f(x)$ at
$x=\frac{n-1}{2}$
when $n$ is odd, and at $x=\frac{n}{2}$
when $n$ is even and $n\not\equiv0\pmod{6}$.


\begin{figure}[!ht]
	\begin{center}
	\includegraphics[width=4cm]{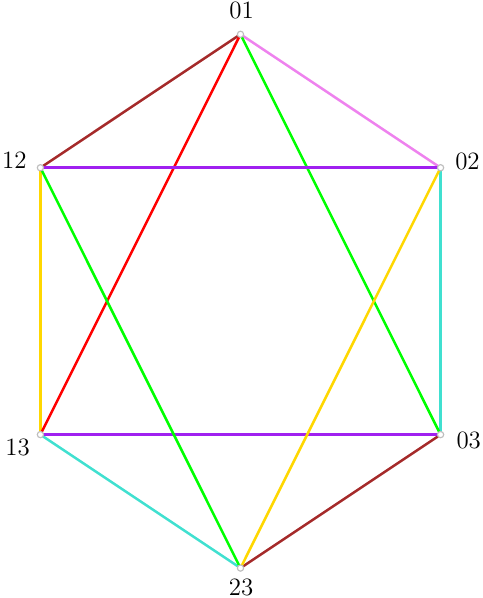}
  \vskip-.3cm
	\caption{An achromatic $7$-edge-coloring of $J(4,2)$. }
		\label{fig:J(4,2)}
	\end{center}
\end{figure}


\section{Conclusion and future research}
\label{sec:conclusion}



As said in the Introduction, the Johnson graphs are a subclass of token graphs, with $G=K_n$ and $F_k(G)=J(n, k)$. In this paper we started on this family, but a natural direction for future research is the study of the achromatic index of token graphs arising from other graph families, such as cycles and paths. 

In general, our objective is to understand how an achromatic edge-coloring of a graph 
$G$ induces an achromatic edge-coloring on the 2-token graph $F_2(G)$ of $G$, exploring its combinatorial properties. 

Another open problem is to generalize these results to $J(n,k)$ for $3\leq k \leq n-3$ and to study the achromatic index of token graphs across different families of graphs. 

Finally, regarding this specific problem, one of the primary objectives is to determine the exact value of $\alpha_2(J(n,2))$. In particular, we are interested in the case $\alpha_2(J(7,2))$.
The bounds obtained in this paper for $J(7,2)$ are
$
35 \leq \alpha_2(J(7,2)) \leq 37.
$
We conjecture that the upper bound can be improved to 36, that is,
$
35 \leq \alpha_2(J(7,2)) \leq 36.
$
We invite the reader to determine the exact value of $\alpha_2(J(7,2))$ or to further improve these bounds using computational or other suitable techniques.
\section{Acknowledgments} 
We express our special thanks to Teresa I. Hoekstra-Mendoza and Miguel Licona-Velázquez for their enthusiasm at the start of this work and for their invaluable help in adapting puzzles to solve the problem.


\end{document}